\UseRawInputEncoding
\documentclass[leqno]{amsart}

\usepackage{stmaryrd,graphicx}
\usepackage{amssymb,mathrsfs,amsmath,amscd,amsthm,color}
\usepackage{float}
\usepackage{mathabx}
\usepackage[all,cmtip]{xy}
\DeclareMathAlphabet{\mathpzc}{OT1}{pzc}{m}{it}
\usepackage{amsfonts,latexsym,wasysym}

\newcommand{\wt}{\widetilde}

\newcommand{\wh}{\widehat}

\newcommand{\fq}{F_{q}}
\newcommand{\stm}{M_{st}}
\newcommand{\tm}{M_{t}}
\newcommand{\fm}{F_{M}}

\newcommand{\bbn}{\mathbb{N}}
\newcommand{\bbq}{\mathbb{Q}}
\newcommand{\bbr}{\mathbb{R}}

\newcommand{\spaces}{\mathbf{Top}}

\newtheorem{theorem}{Theorem}[section]
\newtheorem{lemma}[theorem]{Lemma}
\newtheorem{proposition}[theorem]{Proposition}
\newtheorem{corollary}[theorem]{Corollary}
\theoremstyle{definition}\newtheorem{definition}[theorem]{Definition}
\newtheorem{example}[theorem]{Example}

\newtheorem{remark}[theorem]{Remark}
\newtheorem{problem}[theorem]{Problem}

\begin{document}
\title{Free quasitopological groups}

\author[J. Brazas]{Jeremy Brazas}
\address{West Chester University\\ Department of Mathematics\\
West Chester, PA 19383, USA}
\email{jbrazas@wcupa.edu}

\author[S. Emery]{Sarah Emery}
\address{West Chester University\\ Department of Mathematics\\
West Chester, PA 19383, USA}
\email{se851997@wcupa.edu}


\date{\today}

\begin{abstract}
In this paper, we study the topological structure of a universal construction related to quasitopological groups: the free quasitopological group $F_q(X)$ on a space $X$. We show that free quasitopological groups may be constructed directly as quotient spaces of free semitopological monoids, which are themselves constructed by iterating product spaces equipped with the ``cross topology." Using this explicit description of $F_q(X)$, we show that for any $T_1$ space $X$, $F_q(X)$ is the direct limit of closed subspaces $F_q(X)_n$ of words of length at most $n$. We also prove that the natural map ${\bf i_n}:\coprod_{i=0}^{n}(X\sqcup X^{-1})^{\otimes i}\to F_q(X)_n$ is quotient for all $n\geq 0$. Equipped with this convenient characterization of the topology of free quasitopological groups, we show, among other things, that a subspace $Y\subseteq X$ is closed if and only if the inclusion $Y\to X$ induces a closed embedding $F_q(Y)\to F_q(X)$ of free quasitopological groups.
\end{abstract}

\maketitle

\section{Introduction}

A \textit{quasitopological group} is a group $G$ with a topology such that inverse $g\mapsto g^{-1}$ is continuous and such that the group operation $G\times G\to G$ is continuous in each variable. The second condition is equivalent to the translations $g\mapsto gh$ and $g\mapsto hg$ being homeomorphisms for every $h\in G$. Certainly, every topological group is a quasitopological group. A famous theorem of R. Ellis \cite{ellis} (see also \cite[Theorem 2.3.12]{AT08}) states that every locally compact Hausdorff quasitopological group is a topological group. However, there are many important quasitopological groups, which are not topological groups, including homeomorphism groups $Homeo(X)$ and certain topologized homotopy groups \cite{brazfabelqtop}.

In this paper, we study universal quasitopological groups, which are, in a sense, as far from being a topological group as possible, namely ``free quasitopological groups." The free quasitopological group on a space $X$ is a quasitopological group $F_q(X)$ equipped with a map $\sigma: X\to F_q(X)$, universal in the sense that for every continuous map $f:X\to G$ to a quasitopological group $G$ there is a unique continuous homomorphism $\wt{f}:\fq(X)\to G$ such that $\wt{f}\circ \sigma =f$. In other words, $F_q:\spaces\to \mathbf{qTopGrp}$ is a functor left adjoint to the functor $\mathbf{qTopGrp}\to \spaces$ which forgets the group structure of a quasitopological group.

The free quasitopological groups we investigate are the quasitopological analogues of free topological groups $\fm(X)$ in the sense of Markov \cite{Markov}. Free topological groups hold a place of particular importance in general topological group theory and have an extensive literature. This literature focuses on the case where $X$ is Tychonoff since this is precisely when $\fm(X)$ is Hausdorff and the inclusion of generators $\sigma_1:X\to \fm(X)$ is an embedding. Unfortunately, characterizations of the actual structure of free topological groups, e.g. those in \cite{Sipacheva}, are quite complicated when $X$ lacks certain compactness properties. If $X$ is an inductive limit of a nested sequence of compact Hausdorff subspaces, i.e. a $k_{\omega}$-space, then $\fm(X)$ may be described conveniently is the inductive limit of the subspaces $\fm(X)_n$ of words of length at most $n$ and the natural functions ${\bf i_n}:\coprod_{i=0}^{n}(X\sqcup X^{-1})^n\to \fm(X)_n$ are quotient maps \cite{MMO} (see also \cite{Tkachenko}). In general, both of these properties holding is equivalent to $\fm(X)$ being the natural quotient of the free topological monoid $\tm(X\sqcup X^{-1})=\coprod_{n=0}^{\infty}(X\sqcup X^{-1})^n$ with respect to word reduction \cite[Statement 5.1]{Sipacheva}. Unfortunately, this characterization of $\fm(X)$ as a quotient space will often fail to hold (e.g. if $X=\bbq$ \cite{FOT}) without imposing some kind of global or local compactness condition on $X$. There is a substantial literature on determining when the maps ${\bf i_n}$ are quotient, which is nicely surveyed in \cite[Section 6]{Sipacheva}. There are also analogous investigations for free paratopological groups \cite{EN1,EN2,PRpara}, which were first introduced in \cite{RSTpara}.

In this paper, we show that the free quasitopological group $\fq(X)$ may always be constructed as the natural quotient of the free semitopological monoid $\stm(X\sqcup X^{-1})$ with respect to word reduction. To construct $\stm(X\sqcup X^{-1})$, and consequently $\fq(X)$, without appealing to adjoint functor theorems (as is often done for free topological groups \cite{Po91}), we employ the \textit{cross product} $X\otimes Y$ of spaces. In Section \ref{sectioncrosstopology}, we recall the \textit{cross topology} \cite[p.15]{AT08} used to define $X\otimes Y$. The cross topology is generally much finer than the product topology and is useful for characterizing separate continuity \cite{HW}. In Section \ref{sectionfreesemitopologicalmonoid}, we construct the free semitopological monoid $\stm(X)$ on a space $X$ as the coproduct $\coprod_{i=0}^{n}X^{\otimes i}$ of iterated cross products of $X$ with itself. In Section \ref{sectionfreeqtopgroups}, we show that if we give the free group $F(X)$ the quotient topology with respect to the natural reduction map $R:\stm(X\sqcup X^{-1})\to F(X)$, then the result is precisely a free quasitopological group $\fq(X)$.

Equipped with this construction of $\fq(X)$, Section \ref{sectiontopologyoffqgs} is devoted to proving a general and practical characterization of the topology of $\fq(X)$. Analogous to the free topological group situation, we let $\fq(X)_n$ denote the subspace of $\fq(X)$ consisting of words of length at most $n$ and ${\bf i_n}:\coprod_{i=0}^{n}(X\sqcup X^{-1})^{\otimes i}\to \fq(X)_n$ be the restriction of $R$. Our main result is the following.

\begin{theorem}\label{mainthm}
If $X$ is a $T_1$ space, then
\begin{enumerate}
\item The canonical injection $\sigma_n:X^{\otimes n}\to \fq(X)$ is a closed embedding,
\item $\fq(X)$ has the weak topology with respect to the subspaces $\{\fq(X)_n\}_{n\in\bbn}$,
\item the canonical map ${\bf i_n}:\coprod_{i=0}^{n}(X\sqcup X^{-1})^{\otimes i}\to \fq(X)_n$ is quotient for all $n\in\bbn$.
\end{enumerate}
\end{theorem}

Since cross-products only preserve compactness in trivial situations, the well-known use of the Stone-\v{C}ech compactification $\beta X$, c.f. \cite{HardyMorrisThompson}, appears to be unhelpful for studying $\fq(X)$. Hence, there appears to be a trade-off. While many techniques used to study free topological groups are no longer helpful, Theorem \ref{mainthm}, whose topological-group analogue rarely holds for $\fm(X)$ appears to provide a direct avenue for answering most questions about $\fq(X)$. For example, as part of Theorem \ref{discretetheorem}, we prove the following four conditions are equivalent: $\fq(X)$ is a topological group, $\fq(X)$ is first countable, $\fq(X)$ is discrete, and $X$ is discrete. Hence, either $\fq(X)$ is a discrete group or it is not a topological group. 

Given a Tychonoff space $X$ and a subspace $Y\subseteq X$, it is a fundamental problem to determine when the inclusion $Y\to X$ induces an embedding of free topological groups $\fm(Y)\to \fm(X)$. This problem has been studied extensively \cite{HardyMorrisThompson,Numela,Samuelultra,Uspenskii} and a full characterization was given by O. Sipacheva \cite{Sipachevasubgroups} in terms of extensions of continuous pseudometrics. In Section \ref{sectionsubspaces}, we exploit Theorem \ref{mainthm} to prove the following.

\begin{theorem}\label{mainthm2}
Let $X$ be a $T_1$ space and $Y\subseteq X$. The inclusion $Y\to X$ induces a closed embedding $\fq(Y)\to \fq(X)$ of free quasitopological groups if and only if $Y$ is closed in $X$.
\end{theorem}

We also prove a version of Theorem \ref{mainthm2} where ``closed embedding" is weakened to ``embedding." In particular, we show in Theorem \ref{embedding2} that if $X$ is a Hausdorff sequential space and $Y\subseteq X$, then the inclusion $Y\to X$ induces an embedding $\fq(Y)\to \fq(X)$ if and only if $Y$ is closed in $X$.

\section{The cross topology}\label{sectioncrosstopology}

Although it is possible to define a cross topology for infinite product spaces, we will restrict to finite products. Generally, we will write $X\times Y$ or $\prod_{j\in J}X_j$ to denote the direct product of sets or spaces and $\pi_k:\prod_{j\in J}X_j\to X_k$, $k\in J$ will denote the projection map. We will denote the coproduct (topological sum) of a family of spaces by $\coprod_{j\in J}X_j$.

\begin{definition}
Given spaces $X_1$ and $X_2$, the \textit{cross topology} on the direct product $X_1\times X_2$ consists of sets $U\subseteq X_1\times X_2$ such that $\pi_1(U\cap X_1\times \{y\})$ is open in $X_1$ for all $y\in X_2$ and $\pi_2(U\cap \{x\}\times X_2)$ is open in $X_2$ for all $x\in X_1$. We will denote the set-theoretic product $X_1\times X_2$ equipped with the cross topology as $X_1\otimes X_2$ and refer to this space as the \textit{cross-product of $X_1$ and $X_2$}.
\end{definition}

The cross topology was introduced by Nov\'ak in \cite{Novak} and is generally much finer than the product topology. For example, the set $\{(0,0)\}\cup\{(x,y)\in\bbr^2\mid |y|>2|x|\text{ or }|x|>2|y|\}$ is open in $\bbr\otimes \bbr$ but not in $\bbr\times \bbr$ (see Figure \ref{fig}). Based on this example, one can imagine far more intricate open sets and observe why describing an explicit neighborhood basis at $(0,0)$ in the cross topology is challenging.

\begin{figure}[h]
\centering \includegraphics[height=2in]{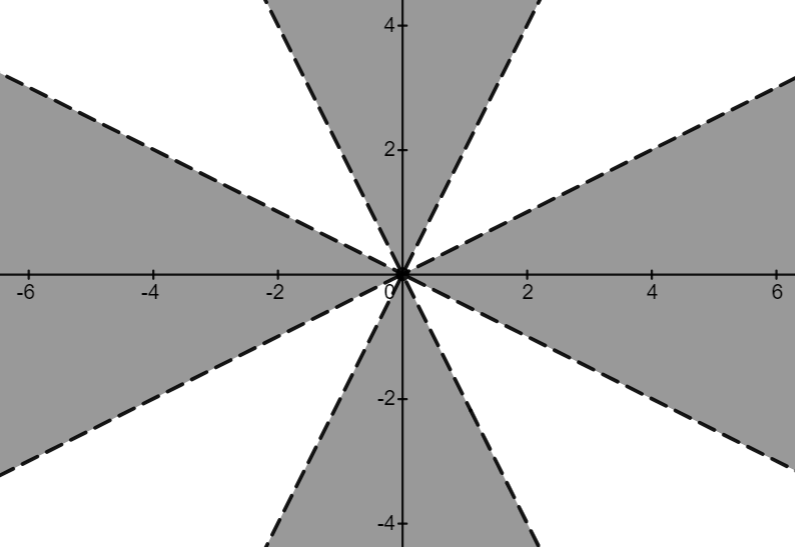}
\caption{\label{fig}An open neighborhood of the origin in $\bbr\otimes\bbr$, which is not open in $\bbr\times \bbr$.}
\end{figure}

\begin{remark}\label{discreteremark}
A space $X$ is $T_1$ if and only if the diagonal $\Delta=\{(x,x)\mid x\in X\}$ is closed in $X\otimes X$. Moreover, if $X$ is $T_1$, then $\Delta$ is a discrete subspace of $X\otimes X$. This follows from the fact that for any $x\in X$, the set $(X\times X\backslash \Delta)\cup \{(x,x)\}$ is open in $X\otimes X$ (see Figure \ref{fig2}). Since the diagonal $\Delta$ is homeomorphic to $X$ as a subspace of $X\times X$, we see that for a $T_1$ space $X$, the cross and product topologies on $X\times X$ agree if and only if $X$ is discrete.

As a consequence, we note that $X\otimes X$ is rarely compact. In fact, $X\otimes X$ is compact Hausdorff if and only if $X$ is discrete and finite. For if $X\otimes X$ is compact Hausdorff, then the identity function $X\otimes X\to X\times X$ is a homeomorphism and the previous paragraph implies that $X$ is discrete. Therefore, $X$ is a compact discrete space and must be finite.
\end{remark}

\begin{figure}[b]
\centering \includegraphics[height=2.2in]{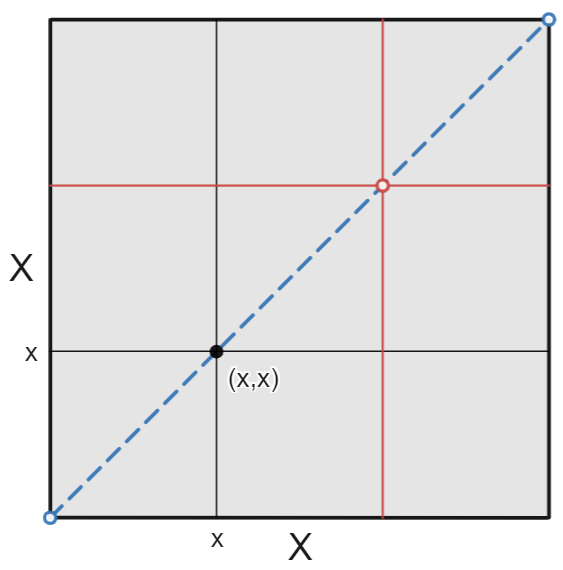}
\caption{\label{fig2} The set $(X\times X\backslash \Delta)\cup \{(x,x)\}$ is open in $X\otimes X$ for all $x\in X$ since each intersection with a projection fiber is open.}
\end{figure}

If we identify the projection fibers $X_1\times \{x_2\}$, $x_2\in X_2$ and $\{x_1\}\times X_2$, $x_1\in X_1$ with the spaces $X_1$ and $X_2$ respectively, then we may view the cross topology as the weak topology with respect to all of these fibers. In other words, the cross topology is the quotient topology with respect to the \textit{fiber decomposition map}
\[\coprod_{x_2\in X_2}X_1\times \{x_2\}\sqcup \coprod_{x_1\in X_1}\{x_1\}\times X_2\to X_1\times X_2,\]
which is the inclusion on each summand.

Recall that a function $f:X_1\times X_2\to Y$ of spaces is \textit{separately continuous} if $f_{x_1}:X_2\to Y$, $f_{x_1}(x_2)=f(x_1,x_2)$ is continuous for all $x_1\in X_1$ and $f_{x_2}:X_1\to Y$, $f_{x_2}(x_1)=f(x_1,x_2)$ is continuous for all $x_2\in X_2$. The primary utility of the cross topology is to very simply characterize separate continuity in terms of ordinary continuity \cite{HW,Novak2}. The following is essentially Proposition 4.1 of \cite{HW}.

\begin{lemma}\label{separatecontlemma}
A function $f:X_1\times X_2\to Y$ is separately continuous if and only if $f:X_1\otimes X_2\to Y$ is continuous. 
\end{lemma}

\begin{proof}
Let $Q:\coprod_{x_2\in X_2}X_1\times \{x_2\}\sqcup \coprod_{x_1\in X_1}\{x_1\}\times X_2\to X_1\otimes X_2$ be the fiber decomposition quotient map. Notice that $f:X_1\times X_2\to Y$ is separately continuous if and only if $f\circ Q$ is continuous. Since $Q$ is quotient, $f:X_1\times X_2\to Y$ is separately continuous if and only if $f:X_1\otimes X_2\to Z$ is continuous.
\end{proof}

It is clear that there is a natural associativity homeomorphism $X_1\otimes (X_2\otimes X_3)\cong (X_1\otimes X_2)\otimes X_3$ and a homeomorphism $X_1\otimes X_2\cong X_2\otimes X_1$. Hence, $\otimes:\mathbf{Top}^2\to\mathbf{Top}$ defines a symmetric monoidal tensor product on the category of topological spaces and continuous functions. We will write $\otimes_{i=1}^{n}X_i$ to denote an $n$-fold cross-product of a sequence of spaces $X_1,X_2,\dots ,X_n$, that is $\prod_{i=1}^{n}X_i$ viewed as an iterated cross-product. A \textit{projection fiber} of $\otimes_{i=1}^{n}X_i$ is a fiber of one of the projection maps $\Pi_j:\otimes_{i=1}^{n}X_i\to \otimes_{i\neq j}X_i$, that is, of the form $\prod_{i=1}^{n}A_i$ where there exists a single $j\in\{1,2,\dots,n\}$ with $A_j=X_j$ and $A_i$ is a singleton when $j\neq i$. By a straightforward induction, it is clear that a set $C\subseteq \otimes_{i=1}^{n}X_i$ is closed (resp. open) if and only if the intersection of $C$ with each projection fiber of $\otimes_{i=1}^{n}X_i$ is closed (resp. open) in that projection fiber. Equivalently, if $PF(X_1,X_2,\dots,X_n)$ is the disjoint union of all projection fibers of $\otimes_{i=1}^{n}X_i$, then the canonical map $PF(X_1,X_2,\dots,X_n)\to \otimes_{i=1}^{n}X_i$ given by inclusion on each summand is a quotient map.

\begin{lemma}\label{productlemma}
If $f_j:X_j\to Y_j$, $j\in\{1,2,\dots,n\}$ are continuous maps, then the cross product function denoted $\otimes_{j=1}^{n}f_j:\otimes_{j=1}^{n}X_j\to \otimes_{j=1}^{n}Y_j$ is continuous. Moreover, $\otimes_{j=1}^{n}f_j$ is a quotient map if and only if $f_j$ is quotient for all $j\in\{1,2,\dots,n\}$. 
\end{lemma}

\begin{proof}
There is a canonical commutative diagram
\[\xymatrix{
PF(X_1,X_2,\dots,X_n) \ar[d] \ar[r]^-{\phi} & PF(Y_1,Y_2,\dots,Y_n) \ar[d]\\
\otimes_{i=1}^{n}X_{i} \ar[r]_-{\otimes_{i=1}^{n}f_i} & \otimes_{i=1}Y_i
}\]where the top map $\phi$ is the restriction of $\otimes_{i=1}^{n}f_i$ on each projection fiber. In particular, such a restriction is the continuous map $f_j$ in one component and constant in all other components. In other words, $\phi$ may be identified with a disjoint union of the maps $f_i$. Therefore, $\phi$ is continuous and is a quotient map if and only if every $f_i$ is quotient. Since both vertical maps are quotient, it follows from the universal property of quotient maps that $\otimes_{j=1}^{n}f_j$ is continuous. Moreover, if each $f_i$ is quotient, then so is $\phi$ and $\otimes_{j=1}^{n}f_j$. The projection maps $\pi_j:\otimes_{i=1}^{n}X_i\to X_j$ are open and therefore are quotient maps. Since $\pi_j\circ \otimes_{i=1}^{n}f_i=f_j\circ \pi_j$, it follows that if $\otimes_{i=1}^{n}f_i$ quotient, then so is each $f_i$.
\end{proof}

\section{Free semitopological monoids}\label{sectionfreesemitopologicalmonoid}

The free monoid on a set $X$ may be represented uniquely up to isomorphism as the monoid of finite words with letters in the set $X$. In particular, $M(X)=\coprod_{n\geq 0}X^n$ where $X^0=\{e\}$ contains the empty word and an $n$-tuple $(x_1,x_2,\dots,x_n)\in X^n$ is represented as a word $w=x_1x_2\dots x_n$. We write $|w|=n$ for the length of such a word, noting that $|e|=0$. The natural operation on $M(X)$ is word concatenation and $e$ is the monoid identity. The monoid structure of $M(X)$ is characterized up to isomorphism as follows: $\rho:X\to M(X)$, $\rho(x)=x$ is the inclusion of free generators and is universal in the sense that for every function $f:X\to N$ to a monoid $N$, there is a unique monoid homomorphism $\wt{f}:M(X)\to N$ such that $\wt{f}\circ\rho=f$. In particular, $\wt{f}$ is defined as $\wt{f}(x_1x_2\dots x_n)=f(x_1)f(x_2)\cdots f(x_n)$ where the product on the right is taken in $N$.

Recall that a \textit{topological monoid} is a monoid $M$ with topology such that the operation $\mu:M\times M\to M$ is continuous. A \textit{semitopological monoid} is a monoid $M$ with topology such that $M\times M\to M$ is separately continuous. According to Lemma \ref{separatecontlemma}, a monoid with topology $M$ is semitopological if and only if $M\otimes M\to M$ is continuous. We will let $\mathbf{TopMon}$ and $\mathbf{STopMon}$ denote the categories of topological and semitopological monoids respectively where in both cases the morphisms are continuous monoid homomorphisms.

When $X$ is a space and each summand $X^n$ is given the product topology, the topological sum $\tm(X)=\coprod_{n\geq 0}X^n$ becomes the free topological monoid on the space $X$. In particular, this results in a functor $\tm:\mathbf{Top}\to\mathbf{TopMon}$ that is left adjoint to the forgetful functor $\mathbf{TopMon}\to\mathbf{Top}$. 

We construct the semitopological analogue using the cross topology. Given a space $X$, we will write $X^{\otimes n}$ for the $n$-fold cross product of $X$ with itself. We define $\stm(X)=\coprod_{n\geq 0}X^{\otimes n}$ and will refer to this as the \textit{free semitopological monoid on }$X$ (this terminology will be justified shortly). Hence, $\stm(X)$ has the same underlying monoid structure as $\tm(X)$ but is equipped with a finer topology.

Recalling that elements of $X^{\otimes n}$ are represented as words $x_1x_2\dots x_n$ with letters in the set $X$, we establish some convenient notation to replace the projection fiber notation above. Given words $w,v\in \stm(X)$, we let $F_{w,v}=wXv=\{wxv\in \stm(X)\mid x\in X\}$. Note that $F_{e,e}=X$, $F_{e,v}=Xv$, and $F_{w,e}=wX$.

For any pair $(w,v)\in \stm(X)^2$, the set $F_{w,v}$ is unique. Hence, $\stm(X)$ is the set-theoretic disjoint union of $\{e\}$ and $F_{w,v}$ ranging over pairs of words $(w,v)\in M(X)^2$. Moreover, for fixed $n\geq 1$, the spaces $F_{w,v}$ such that $|w|+|v|=n-1$, are the projection fibers of $X^{\otimes n}$ and therefore are homeomorphic to $X$. Hence, we have the following.

\begin{lemma}\label{closedinmonoidlemma}
A subset $A\subseteq \stm(X)$ is closed (resp. open) if and only if $A\cap F_{w,v}$ is closed (resp. open) in $F_{w,v}$ for all pairs of words $w,v\in \stm(X)$.
\end{lemma}

Note that $A\cap X^{\otimes 0}$ is either empty or contains a single isolated point so it need not be included in the previous lemma. An equivalent way to state the previous lemma is the following: for all $n\geq 1$, the map $\coprod_{n\geq 1}\coprod_{|w|+|v|=n-1}F_{w,v}\to \stm(X)$ given by the inclusion on each summand is a quotient map.

\begin{theorem}\label{stmonoidtheorem}
For any space $X$, $\stm(X)$ is a semitopological monoid. Moreover, $\rho:X\to \stm(X)$ is continuous and universal in the sense that for every continuous function $f:X\to N$ to a semitopological monoid $N$, there is a unique continuous homomorphism $\wt{f}:\stm(X)\to N$ such that $\wt{f}\circ\rho=f$.
\end{theorem}

\begin{proof}
It suffices to check that the word concatenation operation $\mu:\stm(X)\otimes \stm(X)\to \stm(X)$ is continuous with respect to the cross topology. Since $\otimes$ distributes over topological sums, $\stm(X)\otimes \stm(X)$ is the disjoint union of $X^{\otimes n}\otimes X^{\otimes m}$, $m,n\geq 0$ and the map $\mu$ restricted to this summand is the associativity homeomorphism $X^{\otimes n}\otimes X^{\otimes m}\to X^{\otimes n+m}$. Therefore, concatenation is continuous with respect to the cross topology, proving that $\stm(X)$ is a semitopological monoid.

Certainly, the inclusion map $\rho$ is continuous. Suppose $f:X\to N$ is a continuous function to a semitopological monoid $N$. The unique monoid homomorphism $\wt{f}:\stm(X)\to N$ was described above; it suffices to show that $\wt{f}$ is continuous. Since the operation $\mu_2:N^{\otimes 2}\to N$ is continuous, the associativity of $\otimes$ implies that the $k$-fold operation $\mu_k:N^{\otimes k}\to N$ is continuous. Therefore, $\nu:\stm(N)\to N$ defined as $\mu_k$ on the $k$-th summand is continuous on the disjoint union. Now $\wt{f}=\nu\circ (\coprod_{n\geq 0}f^{\otimes n})$ is the composition of continuous functions and is therefore continuous.
\end{proof}

It follows directly that $\stm:\mathbf{Top}\to \mathbf{STopMon}$ defines a functor that is left adjoint to the forgetful functor $\mathbf{STopMon}\to \mathbf{Top}$. Hence, we refer to $\stm(X)$ as the \textit{free semitopological monoid on} $X$.

\section{Constructing free quasitopological groups}\label{sectionfreeqtopgroups}

We recall the construction of the free group $F(X)$ on a set $X$ as a group of reduced words: let $X^{-1}$ be a homeomorphic copy of $X$ with elements written with superscript $x^{-1}$. Then $F(X)=M(X\sqcup X^{-1})/\mathord{\sim}$ where $\sim$ is the relation generated by identifying $xx^{-1}$ and $x^{-1}x$ with $e$ whenever they appear in a word. Every word is represented by a unique reduced form in which no such cancellations are possible and there is a reduction function $R:M(X\sqcup X^{-1})\to F(X)$ taking a word $w$ to its reduced representative $R(w)$. The operation $(w,v)\mapsto R(wv)$ of word concatenation followed by reduction is a group operation on $F(X)$. The inclusion of free generators $\sigma=R\circ \rho:X\to F(X)$ is universal in the sense the any function $f:X\to G$ to a group $G$ extends uniquely to a group homomorphism $\wh{f}:F(X)\to G$ such that $\wh{f}\circ \sigma=f$. In particular, $\wh{f}$ is given by $\wh{f}=\wt{f^{\pm}}\circ R$ where $\wt{f^{\pm}}:M(X\sqcup X^{-1})\to G$ is the uniquely induced monoid homomorphism induced by the function $f^{\pm}:X\sqcup X^{-1}\to G$ where $f^{\pm}(x)=f(x)$ and $f^{\pm}(x^{-1})=f(x)^{-1}$.

As mentioned in the introduction, if we give $F(X)$ the quotient topology with respect to reduction $R:\tm(X\sqcup X^{-1})\to F(X)$, then $F(X)$ is not always the free topological group on $X$. This quotient topology construction does give $F(X)$ the structure of a quasitopological group that lies somewhere between the free quasitopological group and the free topological group on $X$ (this is studied as the ``reduction topology" in \cite{Br10.1}). A key reason for this failure is that the direct product of two quotient maps is not always a quotient map. We use the contrasting fact from Lemma \ref{productlemma} that the cross product does preserve quotient maps.

\begin{definition}
For a space $X$, let $\fq(X)$ be the free group $F(X)$ with the quotient topology inherited from the reduction function $R:\stm(X\sqcup X^{-1})\to \fq(X)$, that is, as the natural quotient of the free semitopological monoid on $X\sqcup X^{-1}$.
\end{definition}

Combining this definition with Lemma \ref{closedinmonoidlemma}, we have the following practical characterization of the topology of $\fq(X)$.

\begin{lemma}\label{closedingrouplemma}
A set $C\subseteq \fq(X)$ is closed (resp. open) if and only if $R^{-1}(C)\cap F_{w,v}$ is closed (resp. open) for all pairs of words $w,v\in \stm(X\sqcup X^{-1})$.
\end{lemma}

\begin{theorem}
For any space $X$, $\fq(X)$ is a quasitopological group. Moreover, the inclusion of generators $\sigma:X\to \fq(X)$ is continuous and universal in the sense that for any continuous function $f:X\to G$ to a quasitopological group $G$ there is a unique continuous homomorphism $\wh{f}:\fq(X)\to G$ such that $\wh{f}\circ\sigma=f$.
\end{theorem}

\begin{proof}
First, we check that group inversion $in:\fq(X)\to \fq(X)$ is continuous. Consider the reverse function $r:\stm(X\sqcup X^{-1})\to \stm(X\sqcup X^{-1})$, defined by $r(x_{1}^{\epsilon_1}x_{2}^{\epsilon_2}\dots x_{n}^{\epsilon_n})=x_{n}^{-\epsilon_n}\dots x_{2}^{-\epsilon_2}x_{1}^{-\epsilon_1}$ and the following diagram where all vertical maps are quotient and the bottom square commutes. We define a continuous map $g$ making the top square commute.
\[\xymatrix{
\coprod_{w,v\in \stm(X\sqcup X^{-1})}F_{w,v} \ar[d]_{Q} \ar[r]^{g} & \coprod_{w,v\in \stm(X\sqcup X^{-1})}F_{w,v} \ar[d]^-{Q}\\
\stm(X\sqcup X^{-1}) \ar[d]_-{R}  \ar[r]^-{r} & \stm(X\sqcup X^{-1}) \ar[d]^-{R}\\
\fq(X) \ar[r]_-{in} & \fq(X)
}\]
Since all of the vertical maps are quotient, it suffices to define $g$ so that $g$ is continuous and the top square commutes. In particular, for $w,v\in \stm(X\sqcup X^{-1})$, we define $g$ so that its restriction to $F_{w,v}$ is the function $g_{w,v}:F_{w,v}\to F_{r(v),r(w)}$, $g(wx^{\epsilon}v)=r(v)x^{-\epsilon}r(w)$. Since $g_{w,v}$ may be identified with the homeomorphism $X\to X^{-1}$ or $X^{-1}\to X$, it is continuous. Therefore, $g$ is continuous. With this definition, the diagram commutes and it follows that $r$ is continuous. Since all vertical maps in the above diagram are quotient, we conclude that inversion $in$ is continuous.

Next, we must show that the group operation $\gamma:\fq(X)\otimes \fq(X)\to \fq(X)$, $\gamma(w,v)=R(wv)$ is continuous. According to Theorem \ref{stmonoidtheorem}, $\stm(X\sqcup X^{-1})$ is a semitopological monoid, which means that the top map in the following commutative diagram (word concatenation) is continuous. Since the vertical maps are quotient maps (recall the second statement of Lemma \ref{productlemma} for the left map), the bottom map is continuous.
\[\xymatrix{
\stm(X\sqcup X^{-1})\otimes \stm(X\sqcup X^{-1}) \ar[d]_{R\otimes R} \ar[r] & \stm(X\sqcup X^{-1}) \ar[d]^-{R}\\
\fq(X)\otimes \fq(X) \ar[r]_-{\gamma} & \fq(X)
}\]

Since both $R$ and the inclusion $\rho:X\to \stm(X\sqcup X^{-1})$ are continuous, the composition $\sigma=R\circ \rho:X\to \stm(X\sqcup X^{-1})\to \fq(X)$ is continuous. To check the universal property of $\sigma$, let $f:X\to G$ be a map to a quasitopological group $G$. We must establish that the unique homomorphism $\wh{f}:\fq(X)\to G$ satisfying $\wh{f}\circ\sigma=f$ is continuous. Recall that we first extend $f$ to a map $f^{\pm}:X\sqcup X^{-1}\to G$ by $f(x^{-1})=f(x)^{-1}$. Since inversion is continuous in $G$, $f^{\pm}$ is continuous. Since $G$ is a semitopological monoid, Theorem \ref{stmonoidtheorem} ensures that there is a unique continuous monoid homomorphism $\wt{f}:\stm(X\sqcup X^{-1})\to G$ such that $\wt{f}\circ\rho=f^{\pm}$. Since $ \wh{f}\circ R=\wt{f}$ where $R$ is quotient and $\wt{f}$ is continuous, $\wh{f}$ is continuous.
\[\xymatrix{
& \stm(X\sqcup X^{-1}) \ar[d]^-{R} \ar[dr]^-{\wt{f}}\\
X \ar[ur]^-{\rho} \ar@/_2pc/[rr]_-{f} \ar[r]_-{\sigma} & \fq(X) \ar[r]_-{\wh{f}} & G
}\]
\end{proof}

A continuous map $f:X\to Y$ canonically induces a continuous monoid homomorphism $\stm(f\sqcup f^{-1}):\stm(X\sqcup X^{-1})\to\stm(Y\sqcup Y^{-1})$, which in turn induces a continuous group homomorphism $\fq(f):\fq(X)\to\fq(Y)$. Indeed, the continuity of $\fq(f)$ follows from the commutativity of the following square where the vertical maps are quotient. \[
\xymatrix{
\stm(X\sqcup X^{-1}) \ar[d]_-{R} \ar[r]^-{\stm(f)} & \stm(Y\sqcup Y^{-1}) \ar[d]^-{R}\\
\fq(X) \ar[r]_-{\fq(f)} & \fq(Y)
}\]
It follows that $\fq:\mathbf{Top}\to\mathbf{qTopGrp}$ is a functor left adjoint to the forgetful functor $\mathbf{qTopGrp}\to\mathbf{Top}$.

\begin{remark}
Recalling the construction of $\stm(X\sqcup X^{-1})$ using the cross topology, note that we have explicitly constructed $\fq(X)$ as the quotient space of a disjoint union of copies of $X$. Therefore, if $\mathcal{C}$ is a coreflective subcategory of $\mathbf{Top}$, i.e. one which is closed under forming disjoint unions and quotient spaces, then $\fq(X)\in\mathcal{C}$ whenever $X\in\mathcal{C}$. For example, if $X$ is sequential, a k-space, locally connected, or locally path-connected, then so is $\fq(X)$.
\end{remark}

\begin{corollary}\label{quotientmapcor}
If a map $f:X\to Y$ is a quotient map, then the induced homomorphism $\fq(f):\fq(X)\to\fq(Y)$ is a quotient map.
\end{corollary}

\begin{proof}
Recall from Lemma \ref{productlemma} that $\otimes$ preserves quotient maps. Therefore, if $f$ is a quotient map, $\stm(f)$ is a topological sum of cross products of quotient maps and is therefore quotient. Since word reduction maps are quotient by construction, it follows from the last commutative square above that $\fq(f)$ is a quotient map.
\end{proof}

\begin{remark}\label{topgrpremark}
The category $\mathbf{TopGrp}$ of topological groups is a reflective subcategory of $\mathbf{qTopGrp}$, that is for each quasitopological group $G$, there is a topological group $\tau(G)$ and a continuous homomorphism $j: G \to \tau(G)$ such that for every continuous homomorphism $f:G\to H$ to a topological group $H$, there is a continuous homomorphism $g:\tau(G)\to H$ such that $ g\circ j=f$. Moreover, $\tau(G)$ may be constructed as the underlying group of $G$ equipped with a coarser topology (implying that $j$ is the identity homomorphism). This coarser topology may be constructed via transfinite recursion \cite{Br10.1}. The universal property ensures that $\tau(G)=G$ if and only if $G$ is a topological group.

It is straightforward from universal properties that there is a natural isomorphism $\fm(X)\cong\tau(\fq(X))$. Hence, $\fq(X)$ is a topological group if and only if $\fq(X)=\fm(X)$. Certainly, $\fq(X)=\fm(X)$ occurs if $X$ is discrete. We will show in Theorem \ref{discretetheorem} below that for non-discrete, $T_1$ spaces $\fq(X)$ is not a topological group. Hence, $\fq(X)$ is only a topological group in trivial situations.
\end{remark}

Despite the dramatic difference between $\fq(X)$ and $\fm(X)$, there is at least one topological similarity. It is shown in \cite[Cororollary 3.9]{BrazFGasTopGrp} that a quasitopological group $G$ and its topological reflection $\tau(G)$ always share the same lattice of (not necessarily normal) open subgroups. Hence, a subgroup $H\leq F(X)$ is open in $\fq(X)$ if and only if it is open in $\fm(X)$.

\section{The topology of free quasitopological groups}\label{sectiontopologyoffqgs}

For all integers $n\geq 0$, let $\fq(X)_n$ be the subset of $\fq(X)$ consisting of words of length at most $n$. Notice that $\fq(X)_0=\{e\}$ is the trivial subgroup.

\begin{remark}\label{reductionremark}
Within $\stm(X\sqcup X^{-1})$, we will be required to analyze how words of the form $wx^{\epsilon}v\in F_{w,v}$ may be reduced to obtain the unique reduced representative $R(wx^{\epsilon}v)$ in $\fq(X)$. There are several reduction patterns possible. However, either $R(wx^{\epsilon}v)=R(w)x^{\epsilon}R(v)$ or there exists a initial subword $w_1$ of $R(w)$ and terminal subword $v_1$ of $R(v)$ such that $R(wx^{\epsilon}v)=w_1v_1$. The latter case occurs when either the last letter of $R(w)$ or the first letter of $R(v)$ is $x^{-\epsilon}$ (possibly followed by more reduction). Hence, either $R(wx^{\epsilon}v)=R(w)x^{\epsilon}R(v)$ or $|R(wx^{\epsilon}v)|<|R(w)|+|R(v)|$.
\end{remark}

\begin{lemma}\label{fnclosedlemma}
If $X$ is $T_1$ and $n\geq 0$, then $\fq(X)_n$ is closed in $\fq(X)$.
\end{lemma}

\begin{proof}
Fix $n\geq 0$. We must check that $R^{-1}(\fq(X)_n)$ is closed in $\stm(X\sqcup X^{-1})$. Fix words $w,v\in \stm(X\sqcup X^{-1})$. It suffices to show that $C=R^{-1}(\fq(X)_n)\cap F_{w,v}$ is closed in $F_{w,v}$. Note that $C$ consists of words of the form $wx^{\epsilon}v$, $\epsilon\in\{1,-1\}$ for which $|R(wx^{\epsilon}v)|\leq n$. Suppose $wx^{\epsilon}v\in F_{w,v}$ where $|R(wx^{\epsilon}v)|>n$. Recall from Remark \ref{reductionremark} that either $R(wx^{\epsilon}v)=R(w)x^{\epsilon}R(v)$ or $|R(wx^{\epsilon}v)|<|R(w)|+|R(v)|$. In either situation, $n\leq |R(w)|+|R(v)|$.

Since $X$ is $T_1$, \[U=X\backslash\{a\in X\mid a^{\pm 1}\text{ is a letter of }wv\text{ and }a\neq x\}\] is an open neighborhood of $x$ in $X$. Since $U^{\epsilon}$ is open in $X\sqcup X^{-1}$, the set $wU^{\epsilon}v$ is an open neighborhood of $wx^{\epsilon}v$ in $F_{w,v}$. We claim that $wU^{\epsilon}v\cap C=\emptyset$. Let $z\in U$. Then $wz^{\epsilon}v\in wU^{\epsilon}v$. Certainly, if $z=x$, then $|R(wx^{\epsilon}v)|>n$. Hence, we assume $z\neq x$. Note that $R(wz^{\epsilon}v)=R(R(w)z^{\epsilon}R(v))$ and since $z\in U\backslash \{x\}$, the letters $z,z^{-1}$ appear in neither $R(w)$ nor $R(v)$. Hence, $R(wz^{\epsilon}v)=R(w)z^{\epsilon}R(v)$. In particular, $|R(wz^{\epsilon}v)|=|R(w)|+1+|R(v)|\geq n+1$. This proves $C\cap F_{w,v}$ is closed in $F_{w,v}$.
\end{proof}

\begin{corollary}\label{inductivelimitcorollary}
If $X$ is $T_1$, then $\fq(X)=\varinjlim_{n}\fq(X)_n$, that is $C\subseteq \fq(X)$ is closed if and only if $C\cap \fq(X)_n$ is closed in $\fq(X)_n$ for all $n\geq 0$.
\end{corollary}

\begin{proof}
One direction is clear. Suppose that $C\subseteq \fq(X)$ such that $C\cap \fq(X)_n$ is closed in $\fq(X)_n$ for all $n\geq 0$. By Lemma \ref{fnclosedlemma}, $\fq(X)_n$ is closed in $\fq(X)$ and so $R^{-1}(C\cap \fq(X)_n)$ is closed in $\stm(X\sqcup X^{-1})$ for all $n\geq 0$. Fix $w,v\in\stm(X\sqcup X^{-1})$. Then $R^{-1}(C)\cap F_{w,v}=R^{-1}(C\cap \fq(X)_{|w|+|v|+1})\cap F_{w,v}$ is closed in $F_{w,v}$. According to Lemma \ref{closedingrouplemma}, $C$ is closed in $\fq(X)$.
\end{proof}

For each $n\geq 0$, let ${\bf i_n}:\coprod_{i=0}^{n}(X\sqcup X^{-1})^{\otimes i}\to \fq(X)_n$ be the restriction of the quotient map $R$. It is not immediately clear that ${\bf i_n}$ is itself a quotient map since $\fq(X)_n$ has the subspace topology inherited from $\fq(X)$. The desirable property that ${\bf i_n}$ is quotient for a given $n$ often fails in the free topological group setting; as mentioned in the introduction, the intricacies of this failure have been studied extensively.

\begin{lemma}\label{restrictedquotientlemma}
If $X$ is $T_1$, then the subspace topology of $\fq(X)_n$ agrees with the quotient topology with respect to the restricted reduction map ${\bf i_n}:\coprod_{i=0}^{n}(X\sqcup X^{-1})^{\otimes i}\to \fq(X)_n$.
\end{lemma}

\begin{proof}
The statement is clearly true for $n=0$ since $\fq(X)_0$ is a one-point space. We focus on the case when $n\geq 1$. It is easy to see that the quotient topology on $\fq(X)_n$ is finer than the subspace topology inherited from $\fq(X)$. To check that they agree, suppose that $C\subseteq \fq(X)_n$ is closed in the quotient topology with respect to ${\bf i_n}$. By Lemma \ref{fnclosedlemma}, $\fq(X)_n$ is a closed subset of $\fq(X)$. Therefore, it suffices to show that $C$ is closed in $\fq(X)$. To do so, we will verify that $R^{-1}(C)\cap F_{w,v}$ is closed in $F_{w,v}$ for all $w,v\in \stm(X\sqcup X^{-1})$.  Recalling that $n$ is fixed, the assumption that $C$ is closed in the quotient topology on $\fq(X)_n$ implies that whenever $|w|+|v|\leq n-1$, the set ${\bf i}_{\bf n}^{-1}(C)\cap F_{w,v}=R^{-1}(C)\cap F_{w,v}$ is closed in $F_{w,v}$. For $w,v$ with $|w|+|v|\geq n$, we proceed by induction on $|w|+|v|$. 

For our induction hypothesis, we suppose $n\leq m$ and that $R^{-1}(C)\cap F_{w,v}$ is closed in $F_{w,v}$ whenever $m-1= |w|+|v|$. Fix $w,v$ such that $|w|+|v|=m$. Let $wx^{\epsilon}v\in F_{w,v}\backslash R^{-1}(C)$. We consider two cases:
\begin{enumerate}
\item Suppose $|R(w)|+|R(v)|=m$. Since $|R(w)|+|R(v)|=|w|+|v|$, we have $R(w)=w$ and $R(v)=v$. Since $X$ is $T_1$, \[U=X\backslash\{a\in X\mid a^{\pm 1}\text{ is a letter of }wv\text{ and }a\neq x\}\] is an open neighborhood of $x$ in $X$. Now $wU^{\epsilon}v$ is an open neighborhood of $wx^{\epsilon}v$ in $F_{w,v}$. Let $wy^{\epsilon}v\in wU^{\epsilon}v$ such that $y\neq x$. Since the letter $y^{\pm 1}$ appears in neither $w$ nor $v$ and since $w$ and $v$ are already reduced, we have $R(wy^{\epsilon}v)=wy^{\epsilon}v$ and $|wy^{\epsilon}v|=m+1\geq n+1$. Since $wy^{\epsilon}v\notin \fq(X)_n$, we conclude $wy^{\epsilon}v\notin C$. Thus, $wU^{\epsilon}v\cap R^{-1}(C)=\emptyset$.

\item Suppose $|R(w)|+|R(v)|<m$. Then $R(R(w)x^{\epsilon}R(v))=R(wx^{\epsilon}v)\notin C$ and thus $R(w)x^{\epsilon}R(v)\in F_{R(w),R(v)}\backslash R^{-1}(C)$. By the induction hypothesis, there exists an open neighborhood $U$ of $x$ in $X$ such that $R(w)U^{\epsilon}R(v)\cap R^{-1}(C)=\emptyset$. Since $X$ is $T_1$, $V=U\backslash\{a\in X\mid a\text{ is a letter of }wv\text{ and }a\neq x\}$ is an open neighborhood of $x$ in $X$. Now $ wV^{\epsilon}v$ is an open neighborhood of $wx^{\epsilon}v$ in $F_{w,v}$. Suppose $wy^{\epsilon}v\in wV^{\epsilon}v$ where $y\neq x$. Since $w$ and $v$ do not have $y^{\pm 1}$ as a letter, we have $R(wy^{\epsilon}v)=R(w)y^{\epsilon}R(v)\in R(w)V^{\epsilon}R(v)\subseteq R(w)U^{\epsilon}R(v)$. Thus $R(wy^{\epsilon}v)\notin C$. This gives $wV^{\epsilon}v\cap R^{-1}(C)=\emptyset$.
\end{enumerate}
From these two cases, we may conclude that $F_{w,v}\cap R^{-1}(C)$ is closed in $F_{w,v}$ when $|w|+|v|=m$, completing the induction.
\end{proof}

We combine the above results with the following.

\begin{theorem}\label{t1theorem}
The following are equivalent:
\begin{enumerate}
\item $X$ is $T_1$,
\item $\fq(X)$ is $T_1$,
\item $\fq(X)_n$ is closed in $X$ for all $n\geq 0$,
\item $\sigma_n:X^{\otimes n}\to \fq(X)$, $\sigma_n(x_1,x_2,\dots,x_n)=x_1x_2\dots x_n$ is a closed embedding for all $n\geq 1$.
\end{enumerate}
\end{theorem}

\begin{proof}
(1) $\Rightarrow$ (3) is Lemma \ref{fnclosedlemma}. (3) $\Rightarrow$ (2) If the trivial subgroup $\fq(X)_0$ is closed, then $\fq(X)$ is $T_1$ since it is a homogeneous space. (2) $\Rightarrow$ (1) holds since the canonical injection $\sigma:X\to \fq(X)$ is continuous. Hence, (1)-(3) are equivalent. 

(1)  $\Rightarrow$ (4). Suppose $X$ is $T_1$ and $C\subseteq X^{\otimes n}$ is closed. By Corollary \ref{inductivelimitcorollary}, it suffices to show that $\sigma_n(C)\cap \fq(X)_m$ is closed in $\fq(X)_m$ for all $m\geq 0$. Recall that ${\bf i_m}:\coprod_{i=0}^{m}(X\sqcup X^{-1})^{\otimes i}\to \fq(X)_m$ is quotient by Lemma \ref{restrictedquotientlemma}. Note that ${\bf i}_{\bf n}^{-1}(\sigma_n(C))$ is precisely the set $C$ in the summand $X^{\otimes n}$ of $\coprod_{i=0}^{n}(X\sqcup X^{-1})^{\otimes i}$. Therefore, ${\bf i}_{\bf n}^{-1}(\sigma_n(C))$ is closed in $\coprod_{i=0}^{n}(X\sqcup X^{-1})^{\otimes i}$. It follows that $\sigma_n(C)$ is closed in $\fq(X)_n$. Since $\sigma_n(C)\subseteq \fq(X)_n$ and each set $\fq(X)_m$ is closed in $\fq(X)$ for all $m\geq 0$ (Lemma \ref{fnclosedlemma}), we have that $\sigma_n(C)\cap \fq(X)_m$ is closed in $\fq(X)_m$ for all $m\geq 0$.

(4) $\Rightarrow$ (1) Suppose $X$ is not $T_1$. We will show that (4) fails because $Im(\sigma_1)$ is not closed. Choose $x,y\in X$ such that $y\in \overline{\{x\}}\backslash\{x\}$. Let $U$ be an open neighborhood of $w=xyx^{-1}$ in $\fq(X)$. Now $R^{-1}(U)\cap F_{x,x^{-1}}=xVx^{-1}$ for some open neighborhood of $y$ in $X$. Since $y\in \overline{\{x\}}$, we have $x\in V$. Thus $xxx^{-1}\in xVx^{-1}\subseteq R^{-1}(U)$ giving $x=R(xxx^{-1})\in U$. Since every open neighborhood of $xyx^{-1}\notin Im(\sigma_1)$ meets $Im(\sigma_1)$, we conclude that $Im(\sigma_1)$ is not closed.
\end{proof}

\begin{proof}[Proof of Theorem \ref{mainthm}]
Parts (1), (2), and (3) follow from Theorem \ref{t1theorem}, Corollary \ref{inductivelimitcorollary}, and Lemma \ref{restrictedquotientlemma} respectively.
\end{proof}

\begin{corollary}
The converse of Corollary \ref{quotientmapcor} holds when $Y$ is $T_1$.
\end{corollary}

\begin{proof}
Suppose $f:X\to Y$ is a map where $Y$ is $T_1$ and $\fq(f)$ is quotient. Certainly, $f$ must be surjective. Let $Y'$ be the quotient topology on the underlying set of $Y$ inherted from $f$ and $f':X\to Y'$ be the quotient map. By Corollary \ref{quotientmapcor}, $\fq(f)$ is quotient. There is an induced, continuous bijection $i:Y'\to Y$, which induces a continuous group isomorphism $\fq(i):\fq(Y')\to \fq(Y)$ such that $\fq(i)\circ \fq(f')=\fq(f)$. Since both $\fq(f)$ and $\fq(f')$ are quotient, $\fq(i)$ is an isomorphism of quasitopological groups. Since $Y'$ is also $T_1$, Theorem \ref{t1theorem} ensures that the restriction $i=\fq(i)|_{\fq(Y')_1}$ is a homeomorphism. It follows that $f$ is a quotient map.
\end{proof}

\begin{corollary}\label{closedsequencecor}
If $X$ is $T_1$ and $\{w_m\}_{m\in\bbn}$ is a sequence in $\fq(X)$ such that $|w_m|\to\infty$, then $\{w_m\mid m\in\bbn\}$ is closed in $\fq(X)$.
\end{corollary}

\begin{proof}
By assumption, the set $\{w_m\mid m\in\bbn\}\cap \fq(X)_n$ is finite for all $n\in\bbn$. Since $\fq(X)$ is $T_1$ by Theorem \ref{t1theorem}, $\{w_m\mid m\in\bbn\}\cap \fq(X)_n$ is closed in $\fq(X)_n$ for all $n\in\bbn$. It follows from Corollary \ref{inductivelimitcorollary} that $\{w_m\mid m\in\bbn\}$ is closed in $\fq(X)$.
\end{proof}

\begin{corollary}
If $X$ is $T_1$, then $\fq(X)$ is hemicompact with respect to the subsets $\fq(X)_n$, $n\in\bbn$.
\end{corollary}

\begin{proof}
Let $K\subseteq \fq(X)$ be compact and suppose, to obtain a contradiction, there exists a sequence $\{w_m\}_{m\in\bbn}$ in $K$ with $|w_m|\to\infty$. Since $K$ is compact there exists a subsequence $\{w_{m_j}\}_{j\in\bbn}$ that converges to $w\in K$. Since $|w_{m_j}|\to\infty$, we may find $j_0$ such that $|w_{m_j}|>|w|$ for all $j\geq j_0$. However, $\{w_{m_j}\mid j\geq j_0\}$ is closed in $\fq(X)$ by the previous corollary. Therefore, $w\in \{w_{m_j}\mid j\geq j_0\}$; a contradiction.
\end{proof}

The next proposition states that in non-trivial situations, open neighborhoods of the identity in $\fq(X)$ will always contain words of arbitrary length.

\begin{proposition}\label{arblargeprop}
Suppose $X$ is $T_1$ and not discrete. If $U$ is an open neighborhood of $e$ in $\fq(X)$, then $U\cap (\fq(X)_{2m}\backslash \fq(X)_{2m-1})\neq\emptyset$ for all $m\in\bbn$.
\end{proposition}

\begin{proof}
Let $x\in X$ be a non-isolated point of $X$ and let $m\in\bbn$ be arbitrary. Since $R((xx^{-1})^m)=e$, $R^{-1}(U)$ is an open neighborhood of $(xx^{-1})^m$ in $\stm(X\sqcup X^{-1})$. Find an open neighborhood $V_1$ of $x$ such that $V_1x^{-1}(xx^{-1})^{m-1}\subseteq R^{-1}(U)$. Since $x$ is not isolated, we may find $y_1\in V_1\backslash \{x\}$. Since $y_1x^{-1}(xx^{-1})^{m-1}\in R^{-1}(U)$, we may find an open neighborhood $V_2$ of $x$ such that $y_1\notin V_2$ and $y_1V_{2}^{-1}(xx^{-1})^{m-1}\subseteq R^{-1}(U)$. Find $y_2\in V_2$ such that $y_2\neq x$. Then $y_1y_{2}^{-1}(xx^{-1})^{m-1}\in R^{-1}(U)$. Continuing in this way, we may construct a reduced word $w=y_1y_{2}^{-1}y_3y_{4}^{-1}\dots y_{2m-1}y_{2m}^{-1}\in R^{-1}(U)$. Therefore, $w=R(w)\in U$ and $|w|=2m$. 
\end{proof}

The next theorem shows that $\fq(X)$ is only a topological group in trivial situations.

\begin{theorem}\label{discretetheorem}
For any $T_1$ space $X$, the following are equivalent.
\begin{enumerate}
\item $X$ is discrete,
\item $\fq(X)$ is discrete,
\item $\fq(X)$ is a topological group,
\item the square map $sq:\fq(X)\to \fq(X)$, $w\mapsto R(ww)$ is continuous,
\item $\fq(X)$ is first countable.
\end{enumerate}
\end{theorem}

\begin{proof}
The implications (1) $\Leftrightarrow$ (2) $\Rightarrow$ (3) $\Rightarrow$ (4) are clear. To prove, (4) $\Rightarrow$ (1) suppose $sq:\fq(X)\to \fq(X)$ is continuous. Recall from Theorem \ref{t1theorem} that $\sigma_1$ and $\sigma_2$ are embeddings. The restriction of $sq$ to $Im(\sigma_1)$ may be identified with the diagonal map $X\mapsto X\otimes X$. However, as mentioned in Remark \ref{discreteremark}, the diagonal of $X\otimes X$ is a discrete subspace. Since $X$ continuously injects into a discrete space, it must be discrete. This completes the equivalence of (1)-(4). (2) $\Rightarrow$ (5) is clear. To complete the proof, we verify (5) $\Rightarrow$ (1). Suppose, to obtain a contradiction, that $\fq(X)$ is first countable and not discrete. Since $X$ is $T_1$, $\fq(X)$ is $T_1$ by Theorem \ref{t1theorem}. Find a countable basis $U_1\supseteq U_2\supseteq U_3\supseteq \cdots$ in $\fq(X)$ at $e$. According to Proposition \ref{arblargeprop}, we may find $w_n\in U_n$ such that $|w_n|\geq n$. Since $\{U_n\}_{n\in\bbn}$ is a neighborhood base at $e$, we must have $\{w_n\}_{n\in\bbn}\to e$ in $\fq(X)$. However, since $|w_n|\to\infty$, the set $\{w_n\mid n\in\bbn\}$ is closed in $\fq(X)$ by Corollary \ref{closedsequencecor}. Therefore, $e\in \{w_n\mid n\in\bbn\}$; a contradiction.
\end{proof}

It is a well-known theorem of R. Ellis \cite{ellis} that every locally compact Hausdorff quasitopological group is a topological group. Hence, free quasitopological groups are only locally compact in trivial situations.

\begin{corollary}
If $\fq(X)$ is Hausdorff and non-discrete, then $\fq(X)$ is neither first countable nor locally compact.
\end{corollary}

It is possible for quasitopological groups, which are not topological groups, to be constructed as quotients of topological monoids. This is only possible for free quasitopological groups in trivial situations.

\begin{corollary}
If $X$ is $T_1$ and non-discrete, then there does not exist a topological monoid $M$ and a monoid epimorphism $M\to \fq(X)$, which is a topological quotient map.
\end{corollary}

\begin{proof}
If there existed such a quotient map $M\to \fq(X)$, then the continuous square map $M\to M$, $x\mapsto xx$ would imply that the square map $\fq(X)\mapsto \fq(X)$ is continuous; a contradiction of Theorem \ref{discretetheorem}.
\end{proof}

To complete this section, we remark on separation axioms. Recall that a space $X$ is \textit{functionally Hausdorff} if whenever $a,b\in X$ and $a\neq b$, there exists a continuous function $f:X\to\bbr$ such that $f(a)\neq f(b)$. It is a theorem of B.V.S. Thomas that a space $X$ is functionally Hausdorff if and only if $\fm(X)$ is Hausdorff \cite[Theorem 0.1]{Thomas}. Hence, if $X=\bbr^2$ with the irrational-slope topology, then $X$ is Hausdorff but $\fm(X)$ is not. One might suspect that since $\fq(X)$ has a much finer topology and a simpler ``explicit" construction that the analogous statement might be true for free quasitopological groups. Unfortunately, some difficulty remains in determining when $\fq(X)$ is Hausdorff. This difficulty appears to stem from the lack of a convenient basis for the cross topology.

\begin{problem}\label{hausdorffproblem}
If $X$ is Hausdorff, must $\fq(X)$ be Hausdorff?
\end{problem}

Nevertheless, we can apply Thomas' theorem to obtain a characterization for the functionally Hausdorff property.

\begin{theorem}
A space $X$ is functionally Hausdorff if and only if $\fq(X)$ is functionally Hausdorff.
\end{theorem}

\begin{proof}
Since $\sigma:X\to \fq(X)$ is a continuous injection, one direction is clear. If $X$ is functionally Hausdorff, then $\fm(X)$ is Hausdorff by Thomas' theorem. Since all Hausdorff topological groups are Tychonoff, $\fm(X)$ is functionally Hausdorff. Additionally, the topology of $\fq(X)$ is finer than that of $\fm(X)$. It follows that $\fq(X)$ is functionally Hausdorff.
\end{proof}

\section{Free quasitopological groups on subspaces}\label{sectionsubspaces}

When $Y\subseteq X$, the inclusion $Y\to X$ induces a continuous injection $\fq(Y)\to \fq(X)$. It is natural to ask when this is an embedding, i.e. when the subspace topology on $F(Y)$ inherited from $\fq(X)$ agrees with the free quasitopological group topology. This problem has been studied extensively for free topological groups. The following characterization is due to Sipacheva \cite[Theorem 1]{Sipachevasubgroups} who made use of previous work by Pestov \cite{pestovsubgroup}: if $X$ is a Tychonoff space, then the inclusion $Y\to X$ induces an embedding $\fm(Y)\to \fm(X)$ if and only if every bounded continuous pseudometric on $Y$ can be extended to a continuous pseudometric on $X$. Using our characterization of the topology of free quasitopological groups from the previous section, we prove Theorem \ref{mainthm2}, which shows that the situation for free quasitopological groups is far simpler. 

\begin{proof}[Proof of Theorem \ref{mainthm2}]
Let $f:\fq(Y)\to\fq(X)$ denote the continuous homomorphism induced by the inclusion $Y\to X$. If $Y$ is not closed in $X$, then $f(\fq(Y)_1)=f(\fq(Y))\cap \fq(X)_1$ is not closed in $\fq(X)_1$. Hence, $f(\fq(Y))$ is not closed in $\fq(X)$.

For the converse, suppose $Y$ is closed in $X$. Note that by Lemma \ref{fnclosedlemma} and Corollary \ref{inductivelimitcorollary}, it suffices to show that the restriction $f_n:\fq(Y)_n\to \fq(X)_n$ is a closed embedding for all $n\geq 0$. By Theorem \ref{t1theorem}, this is clear for $n=0$ and $n=1$. Fix $n\geq 2$. 

First, we show that the image $Im(f_n)$ is closed in $\fq(X)_n$. Since ${\bf i_n}$ is quotient, we do so by proving that ${\bf i}_{\bf n}^{-1}(Im(f_n))\cap F_{w,v}$ is closed in $F_{w,v}$ whenever $|w|+|v|\leq n-1$. Suppose $wx^{\epsilon}v\in F_{w,v}\backslash {\bf i}_{\bf n}^{-1}(Im(f_n))$. Then ${\bf i_n}(wx^{\epsilon}v)$ must contain a letter from $X\backslash Y$. If $x\in Y$, let $U=X\backslash \{a\in X\mid a\text{ is a letter of }wv\text{ and }a\neq x\}$. If $x\in X\backslash Y$, let $U=(X\backslash Y)\backslash \{a\in X\mid a\text{ is a letter of }wv\text{ and }a\neq x\}$. In either case, $wU^{\epsilon}v$ is an open neighborhood of $wx^{\epsilon}v$ and if $wz^{\epsilon}v\in wU^{\epsilon}v$, then ${\bf i_n}(wz^{\epsilon}v)$ contains a letter from $X\backslash Y$. Thus $wU^{\epsilon}v\cap {\bf i}_{\bf n}^{-1}(Im(f_n))\cap F_{w,v}=\emptyset$. This completes the proof that $Im(f_n)$ is closed in $\fq(X)_n$.

We now prove $f_n$ is a closed embedding for fixed $n\geq 2$. Suppose $C\subseteq \fq(Y)_n$ is closed. Since ${\bf i_n}$ is quotient and $Im(f_n)$ is closed, it suffices to show that ${\bf i}_{\bf n}^{-1}(f_n(C))$ is closed in ${\bf i}_{\bf n}^{-1}(Im(f_n))$. Fix words $w,v$ in $\stm(X\sqcup X^{-1})$ with $|w|+|v|\leq n-1$. We will show that ${\bf i}_{\bf n}^{-1}(f_n(C))\cap F_{w,v}$ is closed in ${\bf i}_{\bf n}^{-1}(Im(f_n))\cap F_{w,v}$. Suppose $wx^{\epsilon}v\in ({\bf i}_{\bf n}^{-1}(Im(f_n))\cap F_{w,v})\backslash {\bf i}_{\bf n}^{-1}(f_n(C))$. Then all letters of $R(wx^{\epsilon}v)$ are from $Y$. We consider two possible cases for the letter $x$.

First, suppose $x\in X\backslash Y$. Let $U=(X\backslash Y)\backslash \{a\in X\mid a\text{ is a letter of }wv\text{ and }a\neq x\}$ and consider $wU^{\epsilon}v$. If $wz^{\epsilon}v\in wU^{\epsilon}v\cap {\bf i}_{\bf n}^{-1}(Im(f_n))$, then the definition of $U$ and the fact that $x^{\pm 1}$ must appear an odd number of times in $wv$, ensures that $z=x$. Since, $wx^{\epsilon}v\notin {\bf i}_{\bf n}^{-1}(f_n(C))$, we have $wU^{\epsilon}v\cap {\bf i}_{\bf n}^{-1}(f_n(C))=\emptyset$.

Lastly, suppose $x\in Y$. Then the reduced words ${\bf i_n}(w)$ and ${\bf i_n}(v)$ must contain only letters from $Y$. Thus ${\bf i_n}(w)x^{\epsilon}{\bf i_n}(v)$ lies in the open set $F_{{\bf i_n}(w),{\bf i_n}(v)}\backslash {\bf i}_{\bf n}^{-1}(C)$ and so there exists an open neighborhood $W$ of $x$ in $Y$ such that ${\bf i_n}(w)W^{\epsilon}{\bf i_n}(v)\cap {\bf i}_{\bf n}^{-1}(C)=\emptyset$. Find an open neighborhood $U$ of $x$ in $X$ such that $U\cap Y\subseteq W$ and $U\cap \{a\in X\mid a\text{ is a letter of }wv\text{ and }a\neq x\}=\emptyset$. Consider the neighborhood $wU^{\epsilon}v\cap {\bf i}_{\bf n}^{-1}(Im(f_n))$ of $wx^{\epsilon}v$ in $F_{w,v}\cap {\bf i}_{\bf n}^{-1}(Im(f_n))$. Suppose $wz^{\epsilon}v\in wU^{\epsilon}v$ with $z\neq x$. Then ${\bf i_n}(wz^{\epsilon}v)={\bf i_n}(w)z^{\epsilon}{\bf i_n}(v)$. Since we must have ${\bf i_n}(wz^{\epsilon}v)\in Im(f_n)$, we have $z\in Y$. Thus $z\in W$. We have ${\bf i_n}(w)z^{\epsilon}{\bf i_n}(v)\in {\bf i_n}(w)W^{\epsilon}{\bf i_n}(v)$ and thus ${\bf i_n}(w)z^{\epsilon}{\bf i_n}(v)\notin {\bf i}_{\bf n}^{-1}(C)$. It follows that $f_n({\bf i_n}(w)z^{\epsilon}{\bf i_n}(v))={\bf i_n}(wz^{\epsilon}v)\notin f_n(C)$. We conclude that $wU^{\epsilon}v\cap {\bf i}_{\bf n}^{-1}(Im(f_n))\subseteq {\bf i}_{\bf n}^{-1}(Im(f_n))\backslash {\bf i}_{\bf n}^{-1}(f_n(C))$.
\end{proof}

If we remove the condition in the statement of Theorem \ref{mainthm2} that the embedding $\fq(Y)\to\fq(Y)$ have \textit{closed} image, then a full characterization is less clear. However, we are able to identify a satisfactory partial solution.

\begin{proposition}\label{embeddingprop}
Suppose that $X$ is Hausdorff. If the inclusion $Y\to X$ induces an embedding $f:\fq(Y)\to \fq(X)$ of free quasitopological groups, then $Y$ is sequentially closed in $X$.
\end{proposition}

\begin{proof}
Suppose $f:\fq(Y)\to \fq(X)$ is an embedding. Then the restriction $\fq(Y)_2\to \fq(X)_2$ is also an embedding. Suppose to obtain a contradiction that $Y$ is not sequentially closed in $X$. Then there exists a sequence $\{y_j\}_{j\in\bbn}$ of distinct elements in $Y$ such that $\{y_j\}_{j\in\bbn}$ converges to a point $x\in X\backslash Y$. Let $B=\{y_{j}y_{k}^{-1}\in \fq(X)_2\mid k>j\}$ and note $e\notin B$. 

Before proceeding, we note that for all $j\in \bbn$, the set $S_j=\{y_k\mid k>j\}$ is closed in $Y$. Indeed, if $S_j$ has a limit point $y\in Y\backslash S_j$, then we may separate $y$ and $x$ in $X$ by disjoint neighborhoods. In particular, we may separate $y$ and a cofinite subset of $S_j$ with disjoint open sets. It follows that we may separate $y$ and $S_j$ by disjoint open sets in $Y$; a contradiction.

First, we show that the identity $e$ is a limit point of $B$ in $\fq(X)_2$. Let $U$ be an open neighborhood of $e$ in $\fq(X)_2$. Now ${\bf i}_{\bf 2}^{-1}(U)$ is an open neighborhood of $xx^{-1}$ in $M=\coprod_{i=0}^{2}(X\sqcup X^{-1})^{\otimes i}$. Find an open neighborhood $V_1$ of $x$ such that $V_1x^{-1}\subseteq {\bf i}_{\bf 2}^{-1}(U)$. Find $j$ such that $y_j\in V_1$. Then ${\bf i}_{\bf 2}^{-1}(U)$ is an open neighborhood of $y_jx^{-1}$ in $M$ and so we can find an open neighborhood $V_2$ of $x$ in $X$ such that $y_jV_{2}^{-1}\subseteq {\bf i}_{\bf 2}^{-1}(U)$. Find $k>j$ such that $y_k\in V_2$. Now $y_jy_{k}^{-1}\in {\bf i}_{\bf 2}^{-1}(U)$ and we have ${\bf i_2}(y_jy_{k}^{-1})=y_jy_{k}^{-1}\in U$. Thus $B\cap U\neq \emptyset$.

To finish the proof, we will show that $e$ is not a limit point of $B$ in $\fq(Y)_2$ by showing that $B$ is closed in $\fq(Y)_2$. It suffices to show ${\bf i}_{\bf 2}^{-1}(B)$ is closed in $\coprod_{i=0}^{2}(Y\sqcup Y^{-1})^{\otimes i}$. Notice that ${\bf i}_{\bf 2}^{-1}(B)=\{y_jy_{k}^{-1}\in Y\otimes Y^{-1}\mid k>j\}$. For fixed $k$, the set $F_{e,y_{k}^{-1}}\cap {\bf i}_{\bf 2}^{-1}(B)$ is finite and therefore closed. For fixed $j$, we have $F_{y_j,e}\cap {\bf i}_{\bf 2}^{-1}(B)=\{y_jy_{k}^{-1}\mid k>j\}$ which corresponds to the closed set $\{y_k\mid k>j\}$ in $Y$ under the canonical homeomorphism $F_{y_j,e}\cong Y$. Thus $F_{y_j,e}\cap {\bf i}_{\bf 2}^{-1}(B)$ is closed in $F_{y_j,e}$. We conclude that ${\bf i}_{\bf 2}^{-1}(B)$ is closed in $\coprod_{i=0}^{2}(Y\sqcup Y^{-1})^{\otimes i}$.
\end{proof}

By combining Theorem \ref{mainthm} with Proposition \ref{embeddingprop}, we obtain the following characterization.

\begin{theorem}\label{embedding2}
Let $X$ be a sequential Hausdorff space and $Y\subseteq X$. Then the inclusion $Y\to X$ induces an embedding $\fq(Y)\to \fq(X)$ of free quasitopological groups if and only if $Y$ is closed in $X$.
\end{theorem}

\begin{example}
Since $\omega=\{1,2,3,\dots\}$ is not sequentially closed in its one-point compactification $\omega+1=\omega\cup\{\omega\}$, the discrete group $\fq(\omega)$ does not topologically embed onto its non-discrete image in $\fq(\omega+1)$. Similarly, $\fq((0,1])$ does not embed onto its image in $\fq([0,1])$.
\end{example}

We conclude by mentioning that, in addition to Problem \ref{hausdorffproblem}, there remain many interesting, unanswered questions. Some questions are analogues of well-studied questions about free topological groups, which have been at least partially resolved. For example: When is a quasitopological subgroup $H\leq \fq(X)$ isomorphic to a free quasitopological group? How does the dimension of $X$ (in the sense of $ind$ or $dim$) relate to the dimension of $\fq(X)$? Are there Tychonoff spaces $X,Y$ for which $\fm(X)\cong \fm(Y)$ and $\fq(X)\ncong \fq(Y)$?


\begin{thebibliography}{99}
\expandafter\ifx\csname url\endcsname\relax

\fi
\expandafter\ifx\csname urlprefix\endcsname\relax

\fi

\bibitem{AT08} Arhangel'skii, Tkachenko, \emph{Topological Groups and Related Structures}. Atlantis Studies in Mathematics, 2008.

\bibitem{brazfabelqtop} J. Brazas, P. Fabel, \emph{On fundamental groups with the quotient topology}, J. Homotopy and Related Structures {\bf 10} (2015) 71-91.

\bibitem{Br10.1} J.~Brazas, \emph{The topological fundamental group and free topological groups}, Topology Appl. {\bf 158} (2011) 779-802.

\bibitem{BrazFGasTopGrp} J.~Brazas, \textit{The fundamental group as a topological group}, Topology Appl. {\bf 160} (2013) 170-188.


\bibitem{EN1} A. Elfard, P. Nickolas, \emph{On the topology of free paratopological groups}, Bull. Lond. Math. Soc. {\bf 44} (2012), no. 6, 1103-1115.

\bibitem{EN2} A. Elfard, P. Nickolas, \emph{On the topology of free paratopological groups II}, Topology Appl. {\bf 160} (2013) 220-229.

\bibitem{ellis} R. Ellis, \emph{Locally compact transformation groups}, Duke Math. J. {\bf 24} (1957), 119-125.



\bibitem{FOT} T.~Fay, E.~Ordman, B.V.S.~Thomas, \emph{The free topological group over the rationals}. Gen. Topol. Appl. {\bf 10} (1979), no. 1, 33-47.


\bibitem{HardyMorrisThompson} J.P.L. Hardy, S.A. Morris, H.B. Thompson, \textit{Applications of the Stone-\v{C}ech compactification to free topological groups}, Proc. Amer. Math. Soc. {\bf 55} (1976), no. 1, 160-164.

\bibitem{HW} M. Henriksen, R.G. Woods, \emph{Separate versus joint continuity: A tale of four topologies}, Topology Appl. {\bf 97} (1999) 175-205.




\bibitem{MMO} J. Mack, S. A. Morris, and E. T. Ordman, \emph{Free topological groups and the projective dimension of locally compact Abelian groups}, Proc. Amer. Math. Soc. {\bf 40} (1973) 303-308.

\bibitem{Markov} Markov, A.A. \emph{On free topological groups}. Izv. Akad. Nauk. SSSR Ser. Mat. 9 (1945) 3-64 (in Russian); English Transl.: Amer. Math. Soc. Transl. {\bf 30} (1950) 11-88; Reprint: Amer. Math. Soc. Transl. {\bf 8} (1962), no. 1, 195-272.


\bibitem{Novak} J. Nov{\'a}k, \emph{Induktion partiell steiger Funktionen}, Math. Ann. {\bf 118} (1942) 449-461.

\bibitem{Novak2} J. Nov{\'a}k, \emph{On some topologies defined by a class of real-valued functions}, General Topology Appl. {\bf 1} (1971) 247-251.

\bibitem{Numela} E.C. Nummela, \emph{Uniform free topological groups and Samuel compactifications}, Topology Appl. {\bf 13} (1982), no. 1, 77-83.

\bibitem{pestovsubgroup} V.G. Pestov, Some properties of free topological groups, Vestnik Moskov. Univ. Mat. Mekh. {\bf 1} (1982) 35-37 (in Russian); English transl.: Moscow Univ. Math. Bull. {\bf 37} (1982), no. 1, 46-49.


\bibitem{Po91} H.~Porst, \emph{On the existence and structure of free topological groups}, Category Theory at Work (1991) 165-176.

\bibitem{PRpara} N. Pyrch, O. Ravsky, \emph{On free paratopological groups}, Mat. Stud. {\bf 25} (2006) 115-125.

\bibitem{RSTpara} S. Romaguera, M. Sanchis, M. Tkachenko, \emph{Free paratopological groups}, Topology Proc. {\bf 27} (2002), 1-28.



\bibitem{Samuelultra} Samuel, P. \emph{Ultrafilters and compactifications of uniform spaces}, Trans. Amer. Math. Soc. {\bf 64} (1948) 100-132.

\bibitem{Sipachevasubgroups} Sipacheva, O.V., \emph{Free topological groups of spaces and their subspaces}, Topology Appl. {\bf 101} (2000), 181-212.

\bibitem{Sipacheva} Sipacheva, O.V., \emph{The Topology of Free Topological Groups}, J. Math. Sci. {\bf 131} (2005), no. 4, 5765-5838.
%
\bibitem{Thomas} Thomas, B.V.S., \emph{Free topological groups}, General Topology and its Appl. {\bf 4} (1974) 51-72.

\bibitem{Tkachenko} M. G. Tkachenko, \emph{Strict collectionwise normality and countable compactness in free topological groups}, Sib. Mat. Zh. {\bf 28} (1987), no. 5, 167-177.

\bibitem{Uspenskii} V.V. Uspenski., \emph{Subgroups of free topological groups}, Dokl. Akad. Nauk SSSR {\bf 285} (1985), no. 5, 1070-1072 (in Russian); English transl.: Soviet Math. Dokl. {\bf 32} (1985), no. 3, 847-849.



\end{thebibliography}
\end{document}